\documentclass[oneside,10pt]{memoir}

\usepackage{amsmath,amsfonts,amsthm}

\usepackage{fullpage}
\usepackage{setspace} 
 
\usepackage{tikz}
\usetikzlibrary{arrows,backgrounds,calc,decorations,fit,positioning}

\usepackage{lmodern}
\usepackage[T1]{fontenc}

\newtheorem{thm}{Theorem}

\newtheorem{lem}[thm]{Lemma}
\newtheorem{prop}[thm]{Proposition}

\theoremstyle{definition}
\newtheorem{defn}[thm]{Definition}

\theoremstyle{remark}
\newtheorem{rem}[thm]{Remark}

\numberwithin{equation}{chapter}

\newcommand{\N}{\mathbb{N}}
\newcommand{\R}{\mathbb{R}}
\newcommand{\T}{\mathbb{T}}
\newcommand{\Z}{\mathbb{Z}}

\DeclareMathOperator*{\image}{Im}

\begin{document}

\begin{titlingpage}
\aliaspagestyle{page}{empty} 
\pagestyle{empty}
\title{Rotational subsets of the circle}
\author{Dr. J. Ramanathan\\
Department of Mathematics \\
Eastern Michigan University \\
Ypsilanti MI 48197}
\date{}
\maketitle
\begin{abstract}\SingleSpacing
	A rotational subset, relative to a continuous transformation $T: \T \to
	\T$ on $\T = \R/\Z$, is a closed, invariant subset of $\T$ that is
	minimal and on which $T$ respects the standard orientation of the unit
	circle. In the case where $T$ is the standard angle doubling map, such
	subsets were studied by Bullet and Sentenac. The case where $T$
	multiplies angles by an integer $d > 2$ was studied by Goldberg and
	Tresser, and Blokh, Malaugh, Mayer, Oversteegen, and Parris.  These
	authors prove that infinite rotational subsets arise as extensions of
	irrational rotations of the unit circle. This paper studies the extent
	to which such results hold for general continuous maps of the circle.
	In particular, we prove the structure theorem mentioned above holds for
	the wider class of continuous transformations $T$ with finite fibers.
	Our methods are more analytic in nature than the works mentioned.  The
	paper concludes with a construction of infinite rotational sets for a
	class of continuous maps that includes examples that are not equivalent
	to the model cases treated previously. 
	
	\ 
	
	\noindent{\textsc 2010 MSC Subject Clasification: 37B05}

	\noindent{\textsc Supported in part by a Faculty Research Fellowship from Eastern Michigan University (Fall 2016).}

	\noindent{\textsc Email: jramanath@emich.edu}

	\noindent{\textsc ORCID: 0000-0001-8560-1160}
\end{abstract}
\end{titlingpage}
 
\mainmatter
\OnehalfSpacing
\setcounter{page}{1}
\setcounter{chapter}{1}
\pagestyle{plain}

\setlength\headheight{12pt}

\section*{Introduction}

In what follows, $\T$ denotes the unit circle with the standard orientation.

\begin{defn}
	Let $X \subset \T$ and $f: X \to X$ be a continuous transformation. The
	map $f$ \emph{preserves cyclic order} if, for any $P, Q, R \in X$ with
	distinct images, the arcs $P\,Q\,R$ and $f(P)\,f(Q)\,f(R)$ have the
	same orientation.
\end{defn}

Now consider a continuous transformation $T:\T \to \T$ and a compact set $X \subseteq \T$. 

\begin{defn}
	The subset $X$ is \emph{rotational} if
	\begin{itemize}
		\item $X$ is invariant, \emph{i.e.} $TX \subseteq X$, 
		\item $X$ is minimal, and 
		\item $T\vert_X$ preserves cyclic order.
	\end{itemize}
\end{defn}

Our objective is to study the structure of
infinite, rotational, proper subsets under fairly general assumptions about the
nature of $T$. The main result is:

\begin{thm}\label{thm:main}
        Let $T: \T \to \T$ be a continuous function with finite fibers and 
        $X \subset \T$ an infinite, rotational, proper subset of $\T$ with 
        respect to this transformation. Then:
        
\begin{itemize}
\item[i.] The dynamical system $(X,T)$ is an extension of an irrational rotation
of the circle.
\item[ii.] The function $\phi: X \to \T$ that realizes this extension has singleton 
	fibers except at countably many points of $\T$. Over these exceptional points, 
	the fibers have cardinality two, corresponding to endpoints of gaps of the 
	set $X$ in $\T$. 
\item[iii.] $(X,T)$ has a unique ergodic measure $\mu$ and $\phi_\ast \mu$ is the 
	standard Lebesgue measure on $\T$. 
\end{itemize}

\end{thm} 

The above theorem was proved in the case when $T: \T \to \T$ is the angle
doubling map by Bullet and Sentenac \cite{bull-sent}. It was established in the
case where $T$ is the standard $d$-fold cover of the unit circle by Goldberg
and Tresser \cite{gold-tress} and by Blokh, Malaugh, Mayer, Oversteegen and
Parris \cite{bl-mal}. These works were motivated by the study of
the action of quadratic dynamical systems in the complement of the Julia set. 

The proof of theorem \ref{thm:main} will be accomplished
over the next two sections. We then revisit the $d$-fold cover case
from our point of view. The last section presents a class of rotational
subsets that include examples which are not conjugate to the previously known cases. 
On the basis of theorem \ref{thm:main} and the examples of the last section, it is natural
to ask if rotational subsets exist for any continuous transformation of $\T$ with 
degree larger than $1$.

\section*{Structure of rotational subsets}

\emph{Henceforth}, unless otherwise stated, we will work on the case where 
\begin{itemize}
	\item $T$ has finite fibers, and 
	\item the rotational set $X$ is an infinite, proper subset of $\T$.
\end{itemize}

Suppose $x_0 \in X$ is an isolated point of such a dynamical system. Minimality
implies that the forward orbit, $\{T^nx_0: n > 0\}$, is dense in
$X$. As a consequence, we have that $T^n x_0 = x_0$ for some positive integer
$n$. The forward orbit $\mathcal{O}$ must then be finite as well as dense.
Therefore, $X = \{T^nx_0: n \ge 0\}$. This contradiction implies that $X$ cannot
have any isolated points. Therefore $X$ is a perfect subset of $\T$.

By conjugating with the appropriate rotation, we can assume that $0 \notin X$. 
Parameterize $\T$ by the unit interval $[0,1[$ and note that
\[
	0 < \alpha = \inf X < \beta = \sup X < 1.
\]

\begin{lem}\label{lem:perturbation} 
	Suppose $a, b \in X$ and $Ta = b$. Then, there are strictly monotone
	sequences $a_n$ and $b_n$ such that $\lim_{n\to\infty} a_n = a$, 
	$\lim_{n\to\infty} b_n = b$ and $Ta_n = b_n$, for all $n \in \mathbb{N}$. 
\end{lem}

\begin{proof} 
	Since $X$ is perfect, one can construct a sequence $a_n \in X$ such
	that $a_n \ne a$ for all $n \in \N$ and $a_n \to a $ as $n \to \infty$.
	By passing to a subsequence, one can further arrange the $a_n$ to be strictly 
	monotonic. Since $T$ has finite fibers, $Ta_n \ne b$ for all but finitely many 
	$n$. Moreover, $Ta_n \to Ta$. By once again passing to a subsequence, if necessary, 
	we may arrange that $Ta_n$ is strictly monotone. 
\end{proof}

\begin{prop}\label{prop:fiber} 
	The fibers of $T|X$ have cardinality at most two.
\end{prop}
 
\begin{proof}
	Suppose, to the contrary, that $x_0, x_1, x_2 \in X$  are three distinct points arranged in
	increasing order with the same image under $T$. So $T x_i = y$ for $i =
	0, 1$ and $2$. By lemma \ref{lem:perturbation}, there are
	strictly monotone sequences $a^{(i)}_n, i = 0,1,2$ such that $a^{(i)}_n
	\to x_i$. Morevoer, $b^{(i)}_n =  T a^{(i)}_n$ are strictly monotone and 
	approach $y$. 
	
	The proof proceeds according to how $b^{(0)}_n$ and $b^{(2)}_n$
	approach $y$.

	\emph{Case $b^{(0)}_n$ and $b^{(2)}_n$ approach $y$ from the same side:}
	Without loss of generality, assume that both sequences approach $y$ from below.
	In this case, we can use the properties of the sequence $b^{(i)}_n$ to find 
	indices $k$ and $l$ so that 
	\[
		b^{(0)}_k < b^{(2)}_l < y
	\]
	and $a^{(0)}_k$, $x_1$ and $a^{(2)}_l$ are in increasing order.

	This means that $T$ changes the cyclic order of the points $a^{(0)}_k$, $x_1$, 
	and $a^{(2)}_l$.

	\emph{Case $b^{(0)}_n \searrow y$ and $b^{(2)}_n \nearrow y$:}

	Choose $n$ sufficiently large so that
	\[
		a^{(0)}_n < x_1 < a^{(2)}_n.
	\]
	But $T a^{(2)}_n < y < T a^{(0)}_n$, so $T$ doesn't preserve cyclic order. 

	\emph{Case $b^{(0)}_n \nearrow y$ and $b^{(2)}_n \searrow y$:}
	By invoking lemma \ref{lem:perturbation} again, we construct a small perturbation of $x_1$, 
        say $a'$, with the property that $Ta' \ne y$. We treat the case where $Ta' < y$ --- 
	the other case is handled similarly. In this situation, there is a sufficiently large 
	$n$ with the property that 
	\[
		Ta' < Ta^{(0)}_n < y < Ta^{(2)}_n.
	\]
        Observe that $T$ doesn't preserve the cyclic order of $a^{(0)}_n, a'$, and $a^{(2)}_n$.

	Thus, in all three cases, we have a contradiction.
\end{proof}

The minimality condition insures that $T$ is surjective. Consequently, we may put
$\alpha' = \max \{x: Tx = \beta\}$ and $\beta' = \min \{x: Tx = \alpha\}$.

\begin{prop} 
	The inequality 
	\[ 
		\alpha < \alpha' < \beta' < \beta 
	\]
	holds. Moreover, $[\alpha',\beta']$ is a gap for $X$, \emph{i.e.} 
	$X \cap ]\alpha',\beta'[ = \emptyset$.
\end{prop}

\begin{proof}
	Suppose that $\alpha' > \beta'$. Minimality ensures that $\beta' >
	\alpha$.  By using the perturbation result (lemma
	\ref{lem:perturbation}), we can find an $x$ near $\alpha$ such that
	$\alpha < Tx < \beta$. From this we have that $T$ reverses the cyclic
	order of the triple $x, \beta', \alpha'$.

	Now, consider an $x \in X$ with $\alpha' < x < \beta'$. By the definition 
	of $\alpha'$ and $\beta'$, $Tx$ must be distinct
	form $\alpha$ or $\beta$. Hence $Tx$ must be strictly between $\alpha$
	and $\beta$. This means that $T$ reverses the cyclic order of $\alpha',
	x, \beta'$. Contradiction.

	Finally, since $X$ has no isolated points, $\alpha' \ne \alpha$ 
	and $\beta' \ne \beta$.
\end{proof}

\begin{lem}\label{lem:increasing}
	Let $x, y \in X$ If $x < y$ and $Tx < Ty$, then for any $z \in X$ with
	$x < z < y$, we must have $Tx \le Tz \le Ty$.
\end{lem}

\begin{proof}
	Failure of the conclusion clearly entails that $T$ reverses the cyclic
	order of the triple $x, z, y$.
\end{proof}

\begin{prop}
	$\alpha < T\beta \le T\alpha < \beta$.
\end{prop}

\begin{proof}
	Note that minimality rules out the possibilities
	that $T\alpha = \alpha$ and $T\beta = \beta$. 
	
	If $T\beta = \alpha$, then by lemma
	\ref{lem:perturbation} there is an $x \in X$ near $\beta$ such that
	$T\alpha > Tx > \alpha$. This means that $T$ reverses the cyclic order
	of the triple $\alpha, x, \beta$. In a similar way, $T\alpha = \beta$
	can also be ruled out. 

	Only the middle inequality remains.  Suppose, to the contrary, that
	$T\beta > T\alpha$. The previous proposition implies that $T\alpha \le
	Tx \le T\beta$, for any $x \in X$ that is strictly between $\alpha$ and
	$\beta$. Thus, $\image T$ is a proper subset of $X$. This violates the
	minimality assumption.
\end{proof}

For any $x_0, x_1 \in [0,1)$, set
\[
	X_{x_0,x_1} = X \cap [x_0,x_1].
\]

\begin{prop} 
	$T|X$ is monotone increasing on the sets $X_{\alpha,\alpha'}$ and 
	$X_{\beta',\beta}$. Moreover, 
	\[
		Tx > x \qquad \forall x \in X_{\alpha,\alpha'}
	\]
	and
	\[
		Tx < x \qquad \forall x \in X_{\beta,\beta'}.
	\]
\end{prop}

\begin{proof}
	Let $x, y \in X \cap [\alpha,\alpha']$ with $x < y$. Suppose $Ty <
	Tx$. Since $y < \beta'$, the definition of $\beta'$ forces $\alpha =
	T\beta' < Ty$.  Thus, $T$ inverts the cyclic order of $x, y, \beta'$.
	This contradiction shows that $T$ is monotonic increasing on
	$[\alpha,\alpha']$. A similar argument applies to $X \cap
	[\beta',\beta]$.

	Next, let $x \in X \cap [\alpha,\alpha']$ and suppose $Tx \le x$. Since
	$T$ has no fixed points in $X$, $Tx < x$. As neither $\alpha$ nor $\alpha'$ have
	this property, $\alpha < x < \alpha'$ and hence $\alpha < T\alpha \le Tx < x$. 
	Lemma \ref{lem:increasing} then implies that $X \cap [\alpha,x]$ is a
	nonempty, proper, closed invariant subset of $X$. Thus $Tx > x$ for $x
	\in X \cap [\alpha,\alpha']$. The last inequality can be verified
	similarly.
\end{proof}

\section*{Coding by irrational rotations of the circle}

By the Krylov-Bogolioubov theorem, there is a Borel probability measure
on $X$ that is invariant under $T$. Fix one such, $\mu$. Regard $\mu$ as a
measure on $[0,1)$ and write $\tilde\Phi$ for its cumulative distribution
function. Since every point of the infinite set $X$ has dense orbit, the invariant 
probability measure $\mu$ cannot include any point masses. Thus $\tilde\Phi$ is 
continuous. As a consequence, $\Phi = \tilde\Phi|_X$ is a continuous, monotone  increasing
map from $X$ to $[0,1]$. Finally, because
\[
	\tilde\Phi(\alpha) = 0 \qquad \text{and} \qquad \tilde\Phi(\beta) = 1
\]
and the fact that $\tilde\Phi$ is locally constant on the complement of $X$, 
$\Phi: X \to [0,1]$ is surjective.

\begin{prop}\label{prop:aboutPhi} 
	\begin{itemize}
		\item If $x_0,x_1 \in X$ and $X \cap (x_0,x_1) = \emptyset$, then 
			$\Phi(x_0) = \Phi(x_1)$.
		\item On the other hand, if $X \cap (x_0,x_1) \ne \emptyset$, then 
			$\Phi(x_0) > \Phi(x_1)$.
	\end{itemize}
\end{prop}

\begin{proof}
	The first statement follows directly from the observation that $\mu$ contains no 
	point masses.

	To prove the second statement, first put $U =  X \cap (x_0,x_1)$ and
	\[
		\nu_n(U,x) = \left| \{k: k = 0,\dots,n-1 \text{\ and \ } T^kx \in U \} \right|.
	\]
	Finally, write $\mathcal{P}: L^2(X,\mu) \to L^2(X,\mu)$ for the projection onto the subspace
	of function left invariant by $T$. 
	By Mean Ergodic Theorem, 
	\[
		\lim_{n\to\infty} \frac {\nu_n(U,x)}n 
	\]
	converges in $L^2(X,\mu)$ to the projection $\mathcal{P}\mathbf{1}_U$.  
	On the other hand, since $U$ is a non-empty 
	open set and $(X,T)$ is a minimal dynamical system, there is an $\epsilon > 0$ such that
	\[
		\liminf_{n\to\infty} \frac {\nu_n(U,x)}n > \epsilon \qquad \text {for all $x \in X$.}
	\]
	(See, for example, proposition 4.7 in \cite{queff}.)
	Hence,
	\[
		\Phi(x_1) - \Phi(x_0) = \mu(U) \ge \|\mathcal{P}\mathbf{1}_U\| \ge \epsilon > 0. 
	\]
\end{proof}

A corollary of this is that any non-empty open subset of $X$ has positive $\mu$-measure.

\begin{prop} 
\begin{itemize} 
	\item[i.] There is an irrational number $\theta_0 \in (0,1)$ such 
		that \[ \Phi(Tx) = \theta_0 + \Phi(x) \qquad \mod 1 \] for 
		all $x \in X$.  
	\item[ii.] If $P, Q$ and $R$ are distinct points of $X \subset \T$ 
		with distinct images under $\Phi$ then the arcs $PQR$ 
		and \newline$\Phi(P)\Phi(Q)\Phi(R)$ both have the same orientation in $\T$.
\end{itemize} 
\end{prop}

\begin{proof}
	Put $\theta_0 = \mu([\beta',\beta])$. If $x \in X \cap
	[\alpha,\alpha']$, then the $T$ invariance of the measure $\mu$ implies
	that
	\begin{align*}
		\Phi(Tx) &= \mu([\alpha,Tx]) \\
		         &= \mu([\alpha,T\beta]) + \mu([T\alpha,Tx) \\
			 &= \mu([T\beta',T\beta]) + \mu([T\alpha,Tx]) \\
			 &= \theta_0 + \mu([\alpha,x]) \\ 
			 &= \theta_0 + \Phi(x).
	\end{align*}
	If $x \in X \cap [\beta',\beta]$, then 
	\begin{align*}
		\Phi(Tx) &= \mu([\alpha,Tx]) \\
		         &= \mu([T\beta',Tx]) \\
			 &= \mu([\beta',x]) \\
			 &= \Phi(x) - \mu([\alpha,\alpha']) \\
			 &= \Phi(x) + \mu([\beta',\beta]) - 1 \\
			 &= \Phi(x) + \theta_0 \qquad \mod 1.
	\end{align*}
	Moreover, $\theta_0$ is irrational. Otherwise, any finite (closed) orbit in $\T$ under
	translation by $\theta_0$ would have a non-dense preimage under $\Phi$ that is $T$ invariant.

	The second statement follows directly from the definition of $\Phi$.
\end{proof}

If we define $\phi:X \to \T$ by $\phi(x) = \exp(2\pi\imath \Phi(x))$ and $\tau:\T \to \T$
by $\tau(z) = \exp(2\pi\imath \theta_0)z$, the above considerations show that $\phi$ is a 
factor map from $(X,T)$ to $(\T,\tau)$.

\begin{lem}\label{lem:aboutX_0} 
	Let $X_0$ be those points $x = X$ with the property that 
	for every $\delta > 0$, the intervals $(x,x + \delta)$ and $(x,x - \delta)$ 
	contain infinitely many points of $X$. Then: $X_0$ is a dense, uncountable 
	subset of $X$.
\end{lem}

\begin{proof}
	A point is in the complement of $X_0$ in $X$ precisely when it is the endpoint of a 
	gap, \emph{i.e.} a maximal open interval contained in $\T \backslash X$. Since there 
	are at most countably many such open intervals, we conclude that $X \backslash X_0$ 
	is countable. As the perfect set $X$ is uncountable, $X_0$ is uncountable as well. 
	Furthermore, each point of $X \backslash X_0$ is nowhere dense. Therefore, $X \backslash X_0$ 
	is nowhere dense. In other words, $X_0$ is dense. 
\end{proof}

\begin{thm} The dynamical system $(X,T)$ is uniquely ergodic. \end{thm}

\begin{proof}
	Let $x_0 \in X_0$. By proposition \ref{prop:aboutPhi} and lemma \ref{lem:aboutX_0},
	\begin{itemize}
		\item $x \in X$ and $x < x_0 \Rightarrow \Phi(x) < \Phi(x_0)$, and
		\item $x \in X$ and $x > x_0 \Rightarrow \Phi(x) > \Phi(x_0)$.
	\end{itemize}
	Therefore, for any $y \in X$ and $n \ge 0$,
\[
	T^n y \le x_0 \qquad \Leftrightarrow \qquad 
	\{ \Phi(y) + n \theta_0 \} \le \Phi(x_0).
\]
(Here the braces denote fractional part.) Consequently, 
\begin{align*}
	\lim_{n\to\infty} \frac {\# \{k: T^k y \le x_0 \text{\ and \ } 0 \le k < n \}}n & \\
	= & \ \lim_{n\to\infty} \frac {\# \left\{k: \{\Phi(y) + n \theta_0\} \le \Phi(x_0) \text{\ and \ } 0 \le k < n \right\}}n \\
	= & \ \Phi(x_0) = \mu( X \cap [\alpha,x_0]).
\end{align*}
The second limit has been evaluated by invoking the Weyl's equidistribution theorem. 
Since such $x_0$ are dense in $X$, the cumulative distribution of the measure $\mu$ is uniquely
determined by $(X,T)$. In other words, $(X,T)$ is uniquely ergodic. 
\end{proof}

With this last result, we have completed the proof of theorem \ref{thm:main}.

\section*{The $z \mapsto z^d$ case}

We now specialize to the case where $T: \T \to \T$ is given by 
\[ 
Tx = d\cdot x \ \ \ \mod 1.  
\] 
In this situation, the inverse image of $0$ consists of the
$d$ points 
\[ 
\xi_k = \frac k{d} \mod 1 \qquad k = 0,...,d-1.  
\] 
In addition,
$T$ has $d-1$ fixed points: 
\[ 
\eta_k = \frac k{d-1} \mod 1 \qquad k = 0,...,d-2.  
\] 
We also set $\xi_d = \eta_{d-1} = 1$. Given our blanket conditions on $X$, none of the 
$\xi_k$ or $\eta_k$ lie in $X$.

Set $I_k = [\xi_k, \xi_{k+1}]$ for $k = 0,...,d-1$ and note that the interior of $I_k$ are 
precisely those points in $\T$ with a unique $d$-adic expansion that starts with 
the digit $k$. Note also that $T I_k = [0,1]$, since $T$ is just the shift map on the 
$d$-adic expansion. Moreover, $T$ is monotonic increasing on the interior of each $I_k$. 
Each closed interval $I_k$ contains a unique fixed point $\eta_k$. The behavior of $T$ at 
these fixed points can be readily determined. In particular, one checks that:
\[
	Tx < x \qquad \text{for $\xi_k < x < \eta_k$}
\]
and
\[
	Tx > x \qquad \text{for $\eta_k < x < \xi_{k+1}$.}
\]

Let $X_1,...,X_\ell$ be the non-empty sets in the list
\[
	X \cap I_0, X \cap I_1, ..., X \cap I_{d-1}.
\]
The indexing can be arranged so $X_i \subset I_{k_i}$ for some $0 \le k_i < d$ with
\[
	k_1 < k_2 < \cdots < k_\ell.
\]
Set 
\[
	\alpha_i = \inf X_i \qquad \text{and} \qquad \beta_i = \sup X_i.
\]
Since $X$ is perfect, $\alpha_i < \beta_i$ for $i = 1,...,\ell$ and $X \cap (\alpha_i,\beta_i)$ 
is non-empty. Consequently, $\Phi(X_i)$ is the closed interval $[\Phi(\alpha_i),\Phi(\beta_i)]$ 
and has positive length (see proposition \ref{prop:aboutPhi}). Because $\mu$ has no mass on 
the open interval $(\beta_i, \alpha_{i+1})$, $\Phi(\beta_i) = \Phi(\alpha_{i+1})$. Set $t_0 = 0$ 
and $t_i = \Phi(\beta_i)$ for $i = 1,...,\ell$. Then, 
\[
	0 = t_0 < t_1 < \cdots < t_\ell = 1
\]
and
\[
	\Phi(X_i) = [t_{i-1},t_i] \qquad \text{for $i = 1,...,\ell$.}
\]
Recall a consequence of our previous analysis: every point $x \in X$ with $Tx > x$
must lie to the left of every point $y \in X$ with $Ty < y$. Thus, each
$X_i$ must lie completely to one side of the unique fixed point in $I_{k_i}$.
Moreover, if $\sup X_i < \eta_{k_i}$ and $\inf X_j > \eta_{k_j}$ then $i > j$
and $k_i > k_j$.  Let $m$ be the last index, $i$ between $1$ and $\ell$ with
$\sup X_i < \eta_{k_i}$. Clearly, $m < \ell$ and $\sup X_m = \alpha'$ and
$\inf X_{m+1} = \beta'$.  Hence, $t_m = \Phi(\alpha_{m+1}) = \Phi(\beta_{m}) = -\theta_0$. 

We next seek to understand the fibers of $\Phi$. Set 
\[
	\mathfrak D_0 = \{\omega \in [0,1]: \omega + n \theta_0 \ne t_i \text{\ for any\ }n \ge 0
	\text{\ and\ } i = 1,...,\ell\}.
\]
(Note that almost every $\omega \in [0,1]$ is in $\mathfrak D_0$.)
The sequence of points $\omega + n \theta_0$ with $n \ge 0$ determines a sequence of intervals
of the form $[t_k,t_{k+1}]$. This, in turn, implies that the base $d$ expansion of any point in 
the fiber of $\omega$ is uniquely determined. Therefore, $\Phi^{-1}(\omega)$ is a singleton.

In summary, this discussion shows how any rotational subset $X$ must arise from 
the symbolic flow of an irrational rotation of $\T$ relative to an appropriate partition.
The next section shows that this process can be reversed.

\section*{The Inverse Process}

Let $\theta_0$ be an irrational number in $\T$ and let $\tau: \T \to \T$ denote 
rotation by $\theta_0$. Consider a partition of $[0,1]$ into 
$\ell \le d$ subintervals with the requirement that one of the interior nodes is $-\theta_0$:
\[
	0 = t_0 < t_1 < \cdots < t_m < t_{m+1} < \cdots < t_\ell = 1
\]
and $t_m = - \theta_0$. Set $J_k = [t_{k},t_{k+1}]$ for $k = 0,...,\ell-1$.  Next, select a 
coding that maps $\{0,...,\ell-1\}$ to the set of digits $\{0,...,d-1\}$. More precisely, 
choose integers $k_0,...,k_{\ell-1}$ that satisfy
\[
	0 \le k_0 < k_1 < \cdots < k_{\ell-1} \le d-1.
\]
We will show that this data, determines a rotational subset of $\T$ that
inverts the process described in previous section.

Let $\mathfrak{D}_0$ be the set of $\omega \in [0,1]$ satisfying 
$\tau^n(\omega) = \omega + n \theta_0 \ne t_i \mod 1$ for all $n \in \N_0$ and $i = 0,...,\ell$. 
In other words, $\mathfrak{D}_0$ consists of those points of 
$\T \backslash \{t_0, t_1, ..., t_{\ell-1}\}$ whose forward orbit doesn't contain any of 
the nodes $t_i$. 

\begin{prop}
\begin{itemize}
	\item[i.] The complement of $\mathfrak{D}_0$ in $\T$ is countable. 
	\item[ii.] For every $t \in \T$, there is an integer $n \ge 0$ with the property
	that $T^n t \in \mathfrak D_0$.
\end{itemize}
\end{prop}

\begin{proof}
	The map $\tau$ is invertible. The complement of $\mathfrak D_0$ is just the
	countable set
	\[
		\{\tau^{-k} t_i: i = 1,...,\ell \text{\ and\ } k \ge 0 \}.
	\]

	For the proof of the second claim, fix $t \in \T$, If the orbit of $t$ hits 
	the $\{ t_i: i = 0,...,\ell\}$ infinitely often, then there must be an 
	index $i$ such that
	\[
		t + n \theta_0 = t_i \mod 1
	\]
	for infinitely many $n \in \N$. This means that there are two distinct, positive 
	integers $n_0, n_1$ with the property that
	\[
		(n_1 - n_0) \theta_0 = 0 \mod 1.
	\]
	But this contradicts the condition that $\theta_0$ is irrational. This argument proves that
	the forward orbit of such a $t$ must eventually lie completely in $\mathfrak D_0$. 
\end{proof}

The trajectory of any point $\omega \in \mathfrak D_0$ can be encoded by an infinite string,
\begin{equation}\label{eqn:about_E}
	E(\omega) = a_0a_1a_2,...
\end{equation}
where $a_n = k_i$ precisely when $\omega + n \theta_0 \mod 1$ is in $J_i$. Note
that $E(\omega + \theta_0)$ is the shift $a_1a_2...$. The
Kronecker approximation theorem implies that each of the digits $k_i$, $i =
0,...,\ell-1$ occurs infinitely often. In particular, we may unambiguously
interpret $E(\omega)$ as the $d$-adic expansion of a unique real number in the
open unit interval $]0,1[$. 

\begin{prop}
The map $E: \mathfrak{D}_0 \to \T$ is a continuous, injective, monotonic increasing and 
\begin{equation}\label{eqn:dynMap}
	E(\tau(\omega)) = T( E(\omega)).
\end{equation}
\end{prop}

\begin{proof}
        Kronecker's theorem also implies injectivity of $E$. Let $\omega$ and $\omega'$
	be distinct points in $\mathfrak D_0$. Set $\delta = \omega - \omega'$
	and note that there must be an $s$ with the property that $s$ and $s' =
	s + \delta$ lie in the interior of different $J_i$'s.  We may then
	choose $n \in \N_0$ with the property that $\tau^n \omega$ and $\tau^n
	\omega' = \tau^n \omega + \delta$ approximate $s$ and $s'$,
	respectively.  If the error is sufficiently small then $\tau^n\omega$
	and $\tau^n\omega'$ will be in different $J_i's$.  Hence the $d$-adic
	encoding of $\omega$ and $\omega'$ are different. Since neither of
	these encodings can end with an infinite string of $d-1$'s, $E(\omega)
	\ne E(\omega')$. 

        Equation \ref{eqn:dynMap} follows directly from equation \ref{eqn:about_E} and
	the ensuing discussion. 

        Let $\omega \in \mathfrak D_0$ and $\omega_n$ a sequence in $\mathfrak D_0$
	that converges to $\omega$.  Then, since $\tau^k \omega$ is in the
	interior of one of the $J_i$, $\tau^k \omega_n$ must eventually have
	this property as well.  Thus, $E$ must be continuous. 
	
        It only remains to prove that $E$ is monotone increasing.  Let $\omega, \omega'
	\in \mathfrak{D}_0$ with $\omega < \omega'$. Write $E(\omega) =
	a_0a_1a_2...$ and $E(\omega') = a'_0a'_1a'_2...$.  Because $E$ is
	injective,  there is a  first index $i \ge 0$ for which $a_i \ne a'_i$.
	If $i = 0$, $\omega < \omega'$ implies $a_0 < a'_0$. This in turn
	yields $E(\omega) < E(\omega')$. If $i > 0$, then $a_{i-1} = a'_{i-1}$
	entails that both $\tau^{i-1}\omega$ and $\tau^{i-1}\omega'$ are in
	interior of the same $J_k$. Since $\tau$ is increasing on the interior
	of any $J_n$, we must have $a_i < a'_i$. As a consequence, $E(\omega) <
	E(\omega')$ in this case as well.
\end{proof}

Write $\mu = E_\ast\mathcal{L}$ for the image of Lebesgue measure on $\T$ under $E$. 
By equation \ref{eqn:dynMap}, $\mu$ is invariant under $T$. Since, $E$ is injective 
and $\mathcal{L}$ is absolutely continuous, $\mu$ has no point masses and hence is 
absolutely continuous. 

Write $X_0$ for the image of $\mathfrak D_0$ under $E$ and set $X = \overline{X_0}$. 
Since $X_0$ is invariant under $T:\T \to \T$, so is its closure $X$. The Borel measure 
$\mu$ is supported on $X$. Define a map $\Phi: X \to \T$ by
\[
	\Phi(x) = \mu([\alpha,x])
\]
where $\alpha = \min X$. If $x = E t$ for some $t \in \mathfrak D_0$, 
\begin{align*}
	\Phi(x) & = \mu([\alpha,x]) \\
		& = \mathcal{L}( E^{-1} [\alpha,x]) \\
		& = \mathcal{L}( [0,t]) = t
\end{align*}
It follows that $E:\mathfrak D_0 \to X_0$ and $\Phi|_{X_0}$ are inverses. 
Thus $\Phi$ is an invariant map from $X_0$ to $\mathfrak D_0$. 
The continuity of $\Phi$ then implies that it is invariant on $X$ as well.

\begin{thm}
	The dynamical system $(X,T)$ is rotational.
\end{thm}

\begin{proof} 
	We can now prove the minimality of $(X,T)$. Let $x \in X$ and
	write $t = \Phi(x)$.  We would like to prove that the orbit of $x$ is
	dense in $X$. We know that there is a positive integer $N$ with the
	property that for all $n \ge N$, $\Phi(T^n x) = t + n \theta_0 \mod 1$
	is in $\mathfrak D_0$. The continuity of $E$ together with the density
	of the set $\{t + n \theta_0: n \ge N \}$ in $\T$ implies that $\{T^n
	x: n \ge N\}$ is dense in $X_0$ and hence in $X$. Consequently, $(X,T)$
	is minimal.  Finally, since $(\T,\tau)$ preserves cyclic order and
	$\Phi$ is monotonic, it is a easy to check that $(X,T)$ preserves
	cyclic order.  It follows that $(X,T)$ is rotational and that
	$(\T,\tau)$ is its canonical factor.
\end{proof}

\section*{Examples for a class of continuous maps}

In this section, we show that infinite rotational sets exist in a certain class
of continuous transformations on $\T$. Members of this class can have
arbitrarily many fixed points and hence will not, in general, be conjugate to
the transformations treated in the previous two sections.

Here as before, we parametrize $\T$ by the half-open interval $[0,1[$. 
Fix an integer $d > 1$, and let 
\[
0 = x_0 < x_1 < \cdots < x_d = 1
\]
be a partition of the unit interval. Consider a continuous transformation $T: \T \to
\T$ that, for each $k = 0,...,d-1$, satisfies
\begin{itemize}
	\item[i.] monotone increasing function on each half-open interval $[x_k,x_{k+1})$, and 
	\item[ii.] $T(x_k) = 0$ and $\lim_{x\to x_{k+1}^-} T(x) = 1$.
\end{itemize}

We set up a standard symbolic encoding for $T$ in terms of infinite strings 
in the alphabet $\mathcal{A} = \{0, 1,...,d-1\}$. In particular, for each
finite (non-empty) word $I = i_1i_2...i_n$, let
\[
	A_I = \{ x \in [0,1[\  : \forall k = 1,...n, f^{k-1}(x) \in [x_{i_k},x_{i_k+1}[ \ \}.
\]

We collect some useful remarks that are easy to check. 

\begin{rem}
	\begin{itemize}
		\item[i.] The $A_I$ are half-open intervals that are closed on the left.
		\item[ii.] If the finite word $I$ is a prefix for the word $J$, then $A_J \subset A_I$.
		\item[iii.] For any fixed $I$ and $n \ge |I|$, the collection 
			\[
				\{ A_J: |J| = n \ \text{and\ $I$ is a prefix for $J$}\ \}
			\]
			is a partition of $A_I$.
		\item[iv.] If $I = i_1i_2...i_n$ and $I' = i_2...i_n$, then 
			\[
				x \in A_I \implies Tx \in A_{I'}.
			\]
		\item[v.] If $I$ precedes $J$ in lexicographical order, then 
			\[
				x \in A_I\ \text{and}\ y \in A_J \implies x < y.
			\]
	\end{itemize}
\end{rem}

Each point $x \in \T$ determines a unique infinite word $\iota(x) \in \mathcal{A}^{\N}$. 
Since, by construction,
\begin{equation}\label{eqn:factmap}
	\iota \circ T = S \circ \iota
\end{equation}
we have that $\mathcal{A}^{\N}$ is a Borel measurable factor of $(\T,T)$.

\begin{rem}\label{rem:lexord}
The lexicographical ordering on $\mathcal{A}^{\N}$ and the usual order on
$[0,1[$ are also compatible with the factor map $\iota$. In particular, for any
$x, y \in [0,1[$, 
\begin{itemize}
	\item[i.] $x < y$ implies $\iota(x) \le \iota(y)$, and
	\item[ii.] $\iota(x) < \iota(y)$ implies $x < y$.
\end{itemize}
\end{rem}

\begin{lem}\label{lem:symbcoding}
	If $I \in \mathcal{A}^{\N}$ 
	doesn't terminate with an infinite sequence of $d-1$'s, there is an $x
	\in [0,1[$ with the property that $\iota(x) = I$.
\end{lem}
\begin{proof}
	Let $I = i_1i_2....$ and set $I_n = i_1...i_n$ for each $n \in
	\N$. It is enough to show that
	\[
		\bigcap_{n=1}^\infty A_{I_n} \ne \emptyset.
	\]
	The hypothesis entails that for any $I_n$ there is an $m > n $ with the
	property that $\overline{A_{I_m}} \subset A_{I_n}$.  This in turn means
	that 
	\[
		\bigcap_{n=1}^\infty A_{I_n} = \bigcap_{n=1}^\infty \overline{A_{I_n}}.
	\]
	The claim follows since a decreasing sequence of bounded, closed
	intervals must have a non-empty intersection.
\end{proof}

\begin{lem}\label{lem:contsymb}
	Let $x_0 \in \T$ be a point with the property that $\iota(x_0)$ doesn't
	terminate with an infinite sequence of $0$'s. Then $x_0$ is a point of
	continuity for the map $\iota: \T \to \mathcal{A}^{\N}$.
\end{lem}
\begin{proof}
	Write $I = i_1i_2...$ for $\iota(x_0)$ and let $I_n =
	i_1...i_n$. The hypothesis entails that $x_0$ does not lie on the
	left endpoint of any of the $A_{I_n}$. In other words, $x_0$ is in the
	interior of all the $A_{I_n}$. Suppose $x_k \to x_0$ as $k \to \infty$.
	For any $n \in \N$, $x_k \in A_{I_n}$ for all sufficiently large $k$.
	For such $k$, $\iota(x_k)$ will have $I_n$ as a prefix. 
\end{proof}

In view of the analysis of the previous section, we may select an infinite
rotational subset $X \subset \T$ for the transformation $T_0: x \mapsto d\cdot
x \mod 1$ that is an extension for an irrational rotation of $\T$ by $\theta_0
\in \R/\Z$.  By mapping each point of $X$ to its $d$-adic expansion, we may
embed $X$ continuously in $\mathcal{A}^{\N}$.  (The transformation $T_0$ is
just the restriction of the standard shift map on $\mathcal{A}^{\N}$ to $X$.

Since no point of $X$ has a $d$-adic expansion that terminates in an infinite
sequence of $d-1$'s, there is an $a \in \T$ with the property that $\iota(a)
\in X$.  The orbit of $a$, 
\[
	\mathcal{O}(a) = \{ T^i a: i \ge 0 \}
\]
is invariant under $T$. By remark \ref{rem:lexord} and equation
\ref{eqn:factmap}, $T$ preserves cyclic order.  One can check that the same facts
extend to the closure $Y = \overline{\mathcal{O}(a)}$. 

\begin{lem}
	Every point of $Y$ is a point of continuiuty for $\iota$.
\end{lem}
\begin{proof}
	By lemma \ref{lem:contsymb}, it suffices to show that for every $y \in
	Y$, $\iota(y)$ does not end with an infinite sequence of $0$'s. If this
	is not the case then by applying $T$ sufficiently many times we obtain
	a point $y \in Y$ for which $\iota(y) = 000...$. In particular,
	the orbit of $y$ lies in the set $\iota^{-1}(000...)$. The
	analysis in the previous section implies that there exist two points
	$u,v \in X$ whose order is reversed by $S$. Since  $\iota(a)$ has dense
	orbit in $X$, there are integers $n, m \in \N$ such that $\iota(T^na)$
	and $\iota(T^ma)$ are so close to $u$ and $v$ that their order is
	reversed by $S$ as well. By remark \ref{rem:lexord}, $T^na$ and $T^ma$
	have thier orders reversed by $T$. Consequently, $T$ does not preserve
	the cyclic order of $y, T^na$ and $T^ma$. This contradiction proves
	that every point of the closed set $Y$ is a point of continuity of
	$\iota$. 
\end{proof}

Now let $Y_0$ be a minimal subset of the dynamical system $(Y,T)$. Since
$\iota|_{Y_0}$ is continuous, $\iota(Y_0)$ is a compact and invariant in the
minimal set $X$. Thus, $\iota(Y_0) = X$. Since $X$ is an infinite factor of
$Y_0$, $Y_0$ is not finite. With a little more care, it is not hard to show
that the rotation number of any $y \in Y_0$ with resprect to $T$ is $\theta$. 

It is natural to ask at this point if rotational subsets of $\T$ always exist
with respect to any continuous transformation of the circle with degree $d > 1$.

\bibliography{../mybib}{}
\bibliographystyle{plain}

\end{document}